\documentclass{amsart}
\usepackage[english]{babel}
\usepackage{amsfonts, amsmath, amsthm, amssymb,amscd,indentfirst}

\usepackage{esint}
\usepackage{graphicx}
\usepackage{epstopdf}
\usepackage{amsmath,amssymb,latexsym,indentfirst}
\usepackage{times}
\usepackage{palatino}

\newtheorem{theorem}{Theorem}[section]

\newtheorem{proposition}{Proposition}[section]

\newtheorem{remark}{Remark}[section]

\def\bp{\begin{proof}}
\def\ep{\end{proof}}

\def\Lc{{\mathcal L}}

\begin{document}

\title[The scalar curvature of the de Sitter-Schwarzschild space]{Deforming the scalar curvature of the de Sitter-Schwarzschild space}

\author{C\'{\i}cero Tiarlos Cruz}
\author{Levi Lopes de Lima}
\author{Jos\'e Fabio Montenegro}
\address{Universidade Federal do Cear\'a (UFC),
Departamento de Matem\'{a}tica, Campus do Pici, Av. Humberto Monte, s/n, Bloco 914, 60455-760,
Fortaleza, CE, Brazil.}
\email{tiarlos@alu.ufc.br, levi@mat.ufc.br, fabio@mat.ufc.br}
\thanks{L. L. de Lima and J. F. Montenegro were partially supported by CNPq/Brazil research grants. C. T. Cruz was supported by a CNPq Scholarship.}

\begin{abstract}
Building upon the work of Brendle, Marques and Neves  on the construction of counterexamples to Min-Oo's conjecture, we exhibit deformations of the de Sitter-Schwarz\-schild space of dimension $n\geq 3$  satisfying the dominant energy condition and agreeing with the standard metric along the event and cosmological horizons, which remain totally geodesic. Our results actually  hold  for generalized  Kottler-de Sitter-Schwarzschild spaces whose cross sections are compact rank one symmetric spaces and indicate that there exists no analogue of the Penrose inequality in the case of positive cosmological constant.  
As an application we construct  solutions of Einstein field equations satisfying the dominant energy condition and being asymptotic to (or agreeing with) the de Sitter-Schwarzschild space-time both at the event horizon and at spatial infinity. 
\end{abstract}

\maketitle

\section{Introduction}\label{intro}

Let $(M,g)$ be a smooth Riemannian manifold of dimension $n\geq 3$ and $\rho:M\to \mathbb R$ a smooth function. We say that the triple $(M,g,\rho)$ 
is a (time-symmetric) {\em initial data set} (IDS) with cosmological constant $\Lambda\in\mathbb R$ if it satisfies the {scalar constraint equation}
\begin{equation}
\label{consteq}
\frac{R_g}{2}-\Lambda=\rho,
\end{equation}
where $R_g$ is the scalar curvature of $g$. The justification for the terminology comes from the well-known fact that the Cauchy development of such an  IDS 
yields a Lorentzian manifold of dimension $n+1$ which is a time-symmetric solution of Einstein field equations with energy density $\rho$ \cite{C-B}. 

An IDS as above is {\em vacuum} if $\rho\equiv 0$. More generally, it satisfies the {\em dominant energy condition} if $\rho\geq 0$. This latter property singles out those IDS which are relevant from the physical viewpoint.
As usual we normalize the cosmological constant so that
\[
\Lambda=\Lambda_{n,\epsilon}=\frac{n(n-1)}{2}\epsilon,\quad \epsilon=0,\pm 1.
\]
Thus, a vacuum has scalar curvature $R_g=n(n-1)\epsilon$. 

A vacuum IDS $(M,g)$ is {\em static }   
if there exists a smooth positive function  $v$ on $M$ satisfying 
\begin{equation}\label{statcond}
\nabla_{g}^2v-(\Delta_{g}v)g-v{\rm Ric}_{g}=0.
\end{equation} 
A triple $(M,g,v)$ as above is called a (static) {\em Killing initial data} (KID). 
It is known that the KID condition is equivalent to the assertion that the Lorentzian metric 
\begin{equation}\label{cauchyexp}
\overline g=-v^2dt^2+g,
\end{equation}
defined on the manifold $\overline M=\mathbb R\times M$, is a solution to the Einstein field equations in vacuum with  cosmological constant  $\Lambda_{n,\epsilon}$: ${\rm Ric}_{\overline g}=n\epsilon\overline g$. Since the time-like vector field $\partial_t$ is Killing, the orthogonal space-like slices  are totally geodesic and isometric to each other, which means that the Cauchy development (\ref{cauchyexp}) of a KID describes a static universe indeed. Thus,  
it is natural to regard KIDs as ground states in the theory. From this perspective, 
a basic question concerning a KID is whether it  admits a deformation into another IDS which satisfies the dominant energy condition and has the same asymptotic behavior at  infinity.  
If no such deformation exists then we say that the given KID is {\em rigid}.  
Since in General Relativity the total energy is defined by means of  a certain surface integral at spatial infinity, rigidity precludes the existence of exotic IDS's with minimal energy within the given class.

This rigidity issue turns out to be an astonishingly difficult problem. For instance, if one naively tries to deform the KID by slightly perturbing the scalar curvature, it is well-known that  (\ref{statcond}) is precisely the obstruction to implementing this procedure \cite{FM} \cite{CoP}. 
On the other hand, all the rigidity statements established so far ultimately rely on considerations of mass-type invariants.  In order to motivate our main results (Theorems \ref{main0} and \ref{main} below) we now briefly recall a few relevant  contributions in this direction; comprehensive accounts on the subject can be found in \cite{B} and \cite{CoP}.

It is obvious that $\mathbb R^n$ endowed with the standard flat metric is a KID with $\epsilon=0$.
Let $g$ be a complete metric on $\mathbb R^n$ with $R_g\geq 0$. Assume further that $g$ decays to the flat metric at infinity with rate $O(|x|^{-\tau})$, $\tau>n-2$.  
Then  the celebrated
positive mass theorem in General Relativity   \cite{SY} \cite{W} implies that $g$ is actually flat. Thus, $\mathbb R^n$ is rigid in the above sense. 

By adapting Witten's spin techniques, Min-Oo \cite{M-O} extended the rigidity to the hyperbolic space $\mathbb H^n$, which is a KID with $\epsilon=-1$. Again, his  result also follows  from appropriate versions of the positive mass theorem in the asymptotically hyperbolic case established afterwards by several authors \cite{W} \cite{CH}.

The  unit hemisphere $\mathbb S^n_+=\{x\in \mathbb S^n; x_{n+1}\geq 0\}$, 
endowed with the standard round metric $g_0$, is a KID with $\epsilon=1$.
This is 
the {\em de Sitter space} in physical terminology. Despite the lack of spatial infinity, the equator $\partial \mathbb S^n_+$  plays the role of a cosmological horizon, so it  makes sense to inquire to what extent a metric satisfying the corresponding dominant energy condition is determined by its behavior along $\partial\mathbb S^n_+$. 
Motivated by the rigidity results mentioned above, Min-Oo conjectured: {\em if 
$g$ is a metric on $\mathbb S^n_+$ such that   $R_g\geq n(n-1 )$, $g=g_0$ on $\partial \mathbb S^n_+$ and $\partial \mathbb S^n_+$ is totally geodesic with respect to $g$,
then $g$ is isometric to $g_0$.}

Given the analogy with the rigidity phenomena displayed above, 
Min-Oo's conjecture was widely expected to be true, with several special cases being established over the years. Thus, it  came as  a great surprise when Brendle-Marques-Neves \cite{BMN} exhibited counterexamples to the conjecture in any dimension. 

\begin{theorem}
\label{bmnfund} \cite{BMN}
There exists a metric $\hat g$ on $\mathbb S^n_+$ with the following properties:
\begin{itemize}
\item $R_{\hat g}>n(n-1)$ everywhere;
\item ${\hat g}=g_0$ on $\partial \mathbb S^n_+$.
\item $\partial \mathbb S^{n}_+$ is totally geodesic with respect to $\hat g$. 
\end{itemize}
\end{theorem}

In fact, there are counterexamples whose geometry agrees with the standard one in a whole neighborhood of the boundary.

\begin{theorem}\label{bmndef}
\cite{BMN} There exists a metric $\hat g$ on $\mathbb S^n_+$ with the following properties:
\begin{itemize}
\item $R_{\hat g}\geq n(n-1)$, with the strict inequality holding somewhere;
\item $\hat g=g_0$ in a neighborhood of $\partial \mathbb S^n_+$.
\end{itemize}
\end{theorem}

Taken together, these theorems show that the de Sitter space is far from being rigid.
In particular, there is no analogue of the positive mass theorem in the case $\Lambda>0$.

Each of the constant curvature KIDs above embeds into a  family of spherically symmetric KIDs of higher energy given by a mass par\-am\-eter $m>0$. 
Moreover, each KID in this family 
carries a compact, outermost minimal inner boundary $\Gamma\subset v^{-1}(0)$, which physically  represents an {\em event horizon} characteristic of a black hole solution.
The above considerations on rigidity  also make sense 
for the elements in this larger family, where now 
we additionally require that the eventual deformation should  preserve the  geometry (both intrinsic and extrinsic) of the horizon. In the asymptotically flat case ($\epsilon=0$), the (exterior) Schwarzschild space, which is the corresponding  KID, is rigid if $3\leq n\leq 7$, a result that follows from the sharp Penrose inequality established for spin manifolds in these dimensions by  Bray-Lee \cite{BrL}, following previous contributions by Bray \cite{Br} and Huisken-Ilmanen \cite{HI}. Thus, the putative rigidity of  Schwarzschild space remains unsettled in dimension $n\geq 8$, except in the graphical category, where it follows from the work of Lam \cite{L} and Huang-Wu \cite{HW}; see also \cite{dLG1}. 
In the asymptotically hyperbolic case $(\epsilon=-1)$, the corresponding KID is the  anti-de Sitter-Schwarzschild space, which is only known to be rigid in two cases.  First, rigidity holds in dimension $n=3$, a consequence of an estimate for a certain normalized volume due to Brendle-Chodosh \cite{BC}; for sufficiently small deformations, this also follows from the work of Ambrozio \cite{Am}, who established a sharp Penrose inequality in this setting. 
Also, it holds in the graphical case for any dimension, as confirmed  by the optimal Penrose inequality proved in \cite{dLG2}. 
Thus,  
the important question of whether the anti-de Sitter-Schwarzschild space is rigid remains wide open if  $n\geq 4$. We remark that similar rigidity statements in the asymptotically locally hyperbolic case follow from the Penrose-type inequalities   appearing in \cite{LN}, \cite{dLG3} and \cite{GWWX}.  

The main results of this note (Theorems \ref{main0} and \ref{main} below) address the  case $\epsilon=1$ in the presence of an event horizon and constitute the exact analogues of Theorems \ref{bmnfund} and \ref{bmndef} in this setting. Our arguments  actually apply to the class of generalized Kottler-de Sitter-Schwarzschild spaces whose cross sections  are compact rank one symmetric spaces with the (properly normalized) Fubini-Study metric,  
ensuring that any such KID fails to be rigid in any dimension $n\geq 3$ and for any admissible value of the mass parameter.
Hence, non-rigidity seems to be a characteristic feature in the $\Lambda >0$ regime. In particular, there is no analogue of the Penrose inequality in this case.  

As in \cite{BMN}, the proofs of Theorems \ref{main0} and \ref{main} rely on the construction of a test function satisfying a pair of inequalities involving the Jacobi operator of the two-component boundary $\partial M$. Once this is accomplished in Propositions \ref{funcchoice} and \ref{newfunc}, the proofs  are completed in Section \ref{proofmain}. 
Finally, in Section \ref{equalinf} we use the constructed metrics as building blocks to provide physically interesting examples of solutions of Einstein field equations with $\Lambda>0$ whose geometry is controlled at spatial infinity.  

\vspace{0.05cm}
\noindent
{\bf Acknowledgements.} The authors thank F. Marques and I. Nunes for valuable discussions.

\section{Statement of results}\label{state}
 
Here we present our main resuts and the class of generalized Kottler-de Sitter-Schwarz\-schild spaces for which they hold. 
We fix an integer $n\geq 3$ and a positive constant, the mass parameter, satisfying
\begin{equation}\label{masspar}
0<m<\frac{(n-2)^{(n-2)/2}}{n^{n/2}}.
\end{equation}
It is easy to check that there are constants $r_{\pm}=r_{\pm}^{(n,m)}$ satisfying
\begin{equation}\label{rrange}
0<r_-<r_*=\left(\frac{n-2}{n}\right)^{1/2}<r_+<1
\end{equation}
so that the function
\[
V(r)={1-r^2-\frac{2m}{r^{n-2}}}
\]
is positive   on the interior $\stackrel{\circ}{I}$ of the interval $I=[r_-,r_+]$ and vanishes on the extremities.  
Now consider  an $(n-1)$-dimensional compact manifold $(N,h)$. We require that $h$ is Einstein with positive scalar curvature, so that after a scaling we may assume that ${\rm Ric}_h=(n-2)h$. It follows that the metric 
\begin{equation}\label{metstat}
g_m=V^{-1}{dr\otimes dr}+r^2h, 
\end{equation}
defined on the manifold $M=\stackrel{\circ}{I}\times N$, has constant scalar curvature $R_{g_m}\equiv n(n-1)$. In fact, if we set $v=\sqrt{V}$ then $(M,g_m,v)$ is a KID.
Notice that if we  
trace  (\ref{statcond}) with respect to $g_m$ we get
\begin{equation}\label{tracestat}
\Delta_{g_m}v+nv=0.
\end{equation}

We still denote by $M$ the manifold 
$I\times N$, the closure of $\stackrel{\circ}{I}\times N$. 
The metric $g_m$ extends smoothly to a metric on this closure,  still denoted by $g_m$. 
The resulting Riemannian manifold $(M,g_m)$ is called a {\em generalized Kottler-de Sitter-Schwarzschild (gKdSS) space} of mass $m>0$. 
If $h$ is the round metric on the unit sphere $\mathbb S^{n-1}$ 
then $(M,g_m)$ is called the {\em de Sitter-Schwarzschild (dSS) space} of mass $m>0$. 
In particular, in the limiting case $m=0$ we recover the standard (round) metric on de Sitter space $\mathbb S^n_+$. Returning to the general case, notice that the induced metric on each slice 
$N_r=\{r\}\times N\subset M$ is $h_r=r^2h$. Moreover, the slice is a totally umbilical hypersurface, with the common principal curvature being ${v(r)}/r$.
In particular, the boundary $\partial M=N_{r_-}\cup N_{r_+}$ is totally geodesic. 
The boundary components $N_{r_-}$ and $N_{r_+}$ are respectively termed the {\em event} and {\em cosmological} horizons. 

Our results hold for a gKdSS space whose cross section $(N,h)$ is a compact rank one symmetric spaces (CROSS). This remarkable class of manifolds comprises the round spheres and certain spaces which can be represented as the base of a Riemannian submersion with totally geodesic fiber a sphere $\mathbb S^q$ of dimension $q\geq 0$ \cite{B} \cite{BoK} \cite{K}. We have $q=0$ for the real projective space $\mathbb RP^{l}$ of dimension $d=l$, $q=1$ for the complex projective space $\mathbb CP^{l}$ of dimension $d=2l\geq 4$, $q=3$ for the quaternionic projective space  $\mathbb HP^{l}$ of dimension $d=4l\geq 8$ and $q=7$ for the Cayley plane $\mathbb OP^2$ of dimension $d=16$. From this representation a CROSS inherits a natural (Fubini-Study) metric $h_{\rm can}$ with sectional curvature range $1\leq K_{h_{\rm can}}\leq 4$, except for the spheres and real projective spaces, which  satisfy $K_{h_{\rm can}}\equiv 1$. Moreover, these spaces are Einstein because they are irreducible as symmetric spaces. If we further normalize the metric so that 
\[
h=\beta^{-1}h_{\rm can}, \quad \beta=\beta_{d,q}=\frac{d-1}{3q+d-1}, \quad d=n-1,
\]   
then ${\rm Ric}_{h}=(d-1)h$.
Thus, with this normalization for the metric these spaces qualify as cross sections for a gKdSS space. 
Notice that $\beta=1$ in the exceptional cases (spheres and real projective spaces). 

Our goal here is to prove the following non-rigidity results.

\begin{theorem}\label{main0}
For each  $n\geq 3$ and for each value of the mass parameter as in (\ref{masspar}) the gKdSS space $M=I\times N$, where $(N,h)$ is a CROSS, carries a metric $\hat g$ with the following properties:
\begin{itemize}
\item $R_{\hat g}>n(n-1)$ everywhere;
\item ${\hat g}=g_m$ along $\partial M$.
\item $\partial M$ is totally geodesic with respect to $\hat g$. 
\end{itemize}
\end{theorem}

\begin{theorem}
\label{main}
For each  $n\geq 3$ and for each value of the mass parameter as in (\ref{masspar}) the gKdSS space $M=I\times N$, where $(N,h)$ is a CROSS,
carries a metric $\hat g$ such that:
\begin{itemize}
 \item $R_{\hat g}\geq n(n-1)$, with the strict inequality holding somewhere;
 \item $\hat g=g_m$ in a neighborhood of $\partial M$.
\end{itemize} 
\end{theorem}

As remarked in the Introduction, these examples rule out an analogue of the Penrose inequality in
the case of 
positive cosmological constant.

\begin{remark}
\label{inunes}
{\rm A recurrent theme in the proofs of Penrose-type inequalities in case $\Lambda\leq 0$ is the construction of suitable geometric foliations connecting the horizon and the sphere at infinity. This is partially motivated by the fact that the spherical slices in the corresponding KIDs are weakly stable as constant mean curvature surfaces.    
For instance, in \cite{Am} 
such a foliation is shown to exist by using  perturbative methods.
In sharp contrast with this, if $n=3$ it turns out that no such foliation by weakly stable constant mean curvature spheres exists for the metrics in Theorem \ref{main0}. The reason is that the horizons are critical for the Hawking mass, which is known to be strictly monotone along the foliation in the (non-vacuum) interior. Similar remarks also apply to the examples in Theorem \ref{main}. We thank I. Nunes for pointing out to us this amazing consequence of our results.}
\end{remark}

\section{Preliminary results}\label{prelim}

For $r_-\leq r\leq r_+$ we denote by $N_r$ the CROSS $N$ of dimension $d=n-1$ endowed with the metric $h_r=r^2h$, where $h=h_1=\beta^{-1}h_{\rm can}$ is the Fubini-Study metric normalized as in the previous section.

The case $m=0$ and $N=\mathbb S^{n-1} $ in Theorem \ref{main} has been treated in \cite{BMN} and a key role in their analysis is played by the Jacobi operator of the boundary, viewed as a minimal hypersurface. Accordingly, we consider
\[
L_{r_\pm}=-\Delta_{h_{r_\pm}}-{\rm Ric}_{g_m}(\nu_{\pm},\nu_{\pm}),
\]
where $\nu_{\pm}$ is the outward unit normal to $N_{r_\pm}\subset M$, $\Delta_{h_{r_\pm}}=r_\pm^{-2}\Delta_h$ is the Laplacian on $N_{r_\pm}$ and  we use here that $\partial M=N_{r_-}\cup N_{r_+}$ is totally geodesic.
From Gauss equation we have
\[
{\rm Ric}_g(\nu_{\pm},\nu_{\pm})=\frac{1}{2}\left(R_{g_m}-R_{h_{r_\pm}}\right).
\]
Using that $R_{g_m}=n(n-1)$ and $R_{h_{r_\pm}}=R_h/r_\pm^ 2=(n-1)(n-2)/r_{\pm}^2$ we end up with
\begin{equation}\label{estop}
L_{r_\pm}=-\Delta_{h_{r_\pm}}+d_{n,r_\pm}, \quad d_{n,r_\pm}=-\frac{n-1}{2}\left(n-\frac{n-2}{r^2_\pm}\right).
\end{equation}

If, as in \cite{BMN}, we take $m=0$ then $\stackrel{\circ}{I}=(0,1)$ and 
\[
L_1=-\Delta_h-(n-1).
\]
These authors  show the existence of a test function $\eta$
on $\mathbb S^{n-1}$ so that 
\begin{equation}\label{poss1} 
L_{1}\eta>0
\end{equation}
and 
\begin{equation}\label{poss2}
\int_{\mathbb S^{n-1}}\eta L_{1}\eta d{\rm vol}_h>0.
\end{equation}
By adapting their  method we prove here a similar result (see Proposition \ref{newfunc} below). The next proposition is the key step in this direction.

\begin{proposition}
\label{funcchoice}
For each CROSS $(N,h)$ of dimension $d=n-1\geq 2$ and for each $r_+$ as in (\ref{masspar}) there exists a function $\eta_{r_+}$ on $N_{r_+}$ so that 
\begin{equation}
\label{posit1}
L_{r_+}\eta_{r_+}>0
\end{equation}
and 
\begin{equation}\label{posit2}
\int_{N_{r_+}}\eta_{r_+} L_{r_+}\eta_{r_+} \,d{\rm vol}_{h_{r_+}}>0.
\end{equation}
\end{proposition}

\begin{proof}
Except for the spheres, with the given normalization a CROSS has diameter equal to $\pi/2\sqrt{\beta}$. Also, for each $0<s<\pi/2\sqrt{\beta}$ it is known that the geodesic sphere of radius $s$ centered at some point has principal curvatures given by $2\sqrt{\beta}\cot 2\sqrt{\beta}s$ and $\cot \sqrt{\beta}s$, with multiplicities $q$ and $d-q-1$, respectively. Therefore, if $\phi=\phi(s)$ is a radially symmetric function, its Laplacian is 
\begin{eqnarray*}
\Delta_h\phi & = & \ddot\phi+\sqrt{\beta}\left((d-q-1)\cot \sqrt{\beta}s+2q\cot 2\sqrt{\beta}s\right)\dot \phi\\
& = & \ddot\phi+\sqrt{\beta}\left((e-1)\cot \sqrt{\beta}s-\frac{2q}{\sin 2\sqrt{\beta}s}\right)\dot \phi, \quad e= d+q,
\end{eqnarray*}
where the dot means derivative with respect to $s$.
In particular, if we take $\phi=g(f)$, where $f(s)=\cos \sqrt{\beta}s$ and
$g:[0,1]\to\mathbb R$ is smooth, we obtain
\[
\beta^{-1}\Delta_h\phi=(1-f^2)g''(f)+\left(\frac{q}{f}-ef\right)g'(f), 
\] 
where the prime denotes derivatives with respect to  $f$. Hence, by taking 
$g(f)=f^{2k}$ we obtain the relation
\begin{equation}\label{key}
-\Delta_h f^{2k}=\lambda^{(h)}_{k}f^{2k}-2k(2k-1+q)\beta f^{2k-2},
\end{equation} 
where $\lambda^{(h)}_k=2k(2k-1+e)\beta$ is the $k^{\rm th}$ eigenvalue of $\Delta_h$\footnote{If $N=\mathbb S^{n-1}$ then $\lambda^{(h)}_k$ is the $2k^{\rm th}$ eigenvalue of $\Delta_h$.}. Also, if $g(f)=f$ we have
\[
\Delta_hf+\beta\left(ef-\frac{q}{f}\right)=0,
\]
which together with $f^2+\beta^{-1}|\nabla_hf|^2=1$ implies 
\begin{equation}\label{reccross}
\int_{N}f^{\alpha+2}{\rm dvol}_h=\frac{\alpha+1+q}{\alpha+1+e}\int_{N}f^{\alpha}{\rm dvol}_h.
\end{equation}

We next consider the function
\[
\psi=\sum_{k=0}^pa_{2k}f^{2k},
\]
where $a_{2k}\in\mathbb R$ and $p\geq 1$ will be specified later. 
A direct computation using (\ref{key}) yields
\begin{eqnarray*}
L_{r_+}\psi & = & \left(\frac{\lambda_{p}^{(h)}}{r_+^2}+d_{n,r_+}\right)a_{2p}f^{2p}\\
& & \quad + \sum_{k=0}^{p-1}\left[
\left(\frac{\lambda_{k}^{(h)}}{r_+^2}+d_{n,r_+}\right)a_{2k}-\beta
\frac{(2k+q+1)(2k+q+2)}{r_+^2}a_{2k+2}
\right]f^{2k}.  
\end{eqnarray*}
We now observe that $\lambda_0^{(h)}/r^2_++d_{n,r_+}=d_{n,r_+}<0$ whereas for $k\geq 1$, 
\[
\frac{\lambda_{k}^{(h)}}{r^2_+}+d_{n,r_+}\geq\frac
{\lambda_{1}^{(h)}}{r^2_+}+d_{n,r_+}=\frac
{2(e+1)\beta}{r^2_+}+d_{n,r_+}\geq 
\frac
{d}{r^2_+}+d_{n,r_+}
>0.
\] 
If $N=\mathbb S^{n-1}$ then $d=n-1$ is the first eigenvalue of the Laplacian, so the estimate holds true as well.
Thus, if we set
$
a_0=-1
$
and 
\begin{equation}\label{recur1ext} a_{2k+2}=\beta^{-1}\frac{\lambda_{k}^{(h)}+r^2_+d_{n,r_+}}{(2k+q+1)
(2k+q+2)}a_{2k}, \quad k\geq 0, 
\end{equation}
then  $a_{2k}>0$ for $k\geq 1$ and 
\[
L_{r_+}\psi  = \left(\frac{\lambda_{p}^{(h)}}{r_+^2}+d_{n,r_+}\right)a_{2p}f^{2p}\geq 0.
\]

On the other hand, 
\[
\int_{N_{r_+}}\psi L_{r_+}\psi{\rm dvol}_{h_{r_+}}=r^{n-1}_+
\left(\frac{\lambda_{p}^{(h)}}{r_+^2}+d_{n,r_+}\right)a_{2p}\sum_{k=0}^pa_{2k}\int_{N}f^{2k+2p}{\rm dvol}_h.
\]
But (\ref{recur1ext}) implies 
\[
a_{2k}=\beta^{-k+1}\left(\Pi_{i=1}^{k-1}
\frac{\lambda^{(h)}_{i}+r_+^2d_{n,r_+}}{(2i+q+1)(2i+q+2)}\right)a_2,
\]
and (\ref{reccross})
gives
\[
\int_{N}f^{2p+2k}{\rm dvol}_h=\left(
\Pi_{j=1}^k\frac{2p+2j+q-1}{2p+2j+e-1}\right)
\int_{N}f^{2p}{\rm dvol}_h.
\]
Hence,
\begin{eqnarray*}
\int_{N_{r_+}}\psi L_{r_+}\psi{\rm dvol}_{h_r}
& = & 
r^{n-1}_+
\left(\frac{\lambda_{p}^{(h)}}{r_+^2}+d_{n,r_+} \right) \times\\
& & \quad \times
a_{2p}\left(-1+a_2\sum_{k=1}^p\Omega_{k,p}\right)\int_{N}f^{2p}{\rm dvol}_h,
\end{eqnarray*}
where $a_2=-\beta^{-1}r^2_+d_{n,r_+}/(q+1)(q+2)>0$ and
\[ 
\Omega_{k,p}=
\beta^{-k+1}\Pi_{i=1}^{k-1}
\frac{\lambda^{(h)}_{i}+r_+^2d_{n,r_+}}{(2i+q+1)(2i+q+2)}
\,\,\Pi_{j=1}^k 
\frac{2p+2j+q-1}{2p+2j+e-1}.
\] 
Even though $a_2\to 0$ as $r_+\to r_*$, we claim that for each $r_+$ as above, $p$ can be chosen large enough (depending on $n$) so that 
\begin{equation}\label{claim2}
\sum_{k=1}^p\Omega_{k,p}>a_2^{-1}.
\end{equation}
This gives 
\[
\int_{N_{r_+}}\psi L_{r_+}\psi d{\rm vol}_{h_{r_+}}>0,
\]
which allows us to complete the proof by taking $\eta_{r_+}=\psi-c$, where $c>0$ is sufficiently small.

It remains to prove (\ref{claim2}). 
Since $r_+^2d_{n,r_+}> -d$ we have $\Omega_{k,p}\geq b_kc_{k,p}$, where 
\[
b_k=\beta^{-k+1}\Pi_{i=1}^{k-1}
\frac{2i(2i+e-1)\beta-d}{(2i+q+1)(2i+q+2)},
\]
and 
\[
c_{k,p}=\Pi_{j=1}^k\frac{2p+2j+q-1}{2p+2j+e-1}.
\]
Since the sequence
\[
j\mapsto \frac{2p+2j+q-1}{2p+2j+e-1}
\]
is strictly increasing 
we obtain
$
c_{k,p}\geq \gamma_{p}^p
$,
where 
\[
\gamma_{p}=\frac{2p+q+1}{2p+
e+1}<1.
\]
Hence,  
\begin{equation}\label{crucest}
\sum_{k=1}^p\Omega_{k,p}\geq \gamma_{p}^p\sum_{k=1}^pb_k.
\end{equation}
But  
\begin{eqnarray*}
\frac{b_{k+1}}{b_k} & = & \beta^{-1}\frac{2k(2k+e-1)\beta-d}{(2k+q+1)
(2k+q+2)}
\\
& =  &  1-\frac{\alpha}{k}+O\left(\frac{1}{k^2}\right),
\end{eqnarray*}
where
\[
\alpha=\frac{4-e+q}{2}=\frac{5-n}{2}\leq 1.
\]
Thus, $\sum_kb_k=+\infty$ by Gauss test \cite[Theorem 2.12]{BK}.
Since
\[
\lim_{p\to +\infty}\gamma_{p}^p=\exp\left(\frac{q-e}{2}\right),
\]
(\ref{claim2}) follows from (\ref{crucest}).

The remaining case, $N=\mathbb S^{n-1}$, is obtained by repeating the argument above with $q=0$ and $f(s)=\cos s$, $0\leq s\leq \pi$.
\end{proof}

\begin{proposition}
\label{newfunc}
Under the conditions of Proposition \ref{funcchoice}    there exists a function $\eta$ on $\partial M$ so that 
\begin{equation}
\label{posit11}
L_{r_+}\eta>0\,\,\,{\rm on}\,\,\,N_{r_+}, \quad L_{r_-}\eta>0\,\,\,{\rm on}\,\,\,N_{r_-},
\end{equation}
and 
\begin{equation}\label{posit22}
\int_{N_{r_+}}\eta L_{r_+}\eta \,d{\rm vol}_{h_{r_+}}>0, \quad \int_{N_{r_-}}\eta L_{r_-}\eta \,d{\rm vol}_{h_{r_-}}>0. 
\end{equation}
\end{proposition}

\begin{proof}
It follows from (\ref{rrange}) that 
$d_{n,r_-}>0$, 
so we see that for {\em any} nonzero function $\varphi$ on $N_{r_-}$ there holds
\[
\int_{N_{r_-}}\varphi L_{{r_-}}\varphi d{\rm vol}_{h_{r_-}}>0.
\]
Of course, this just expresses the well-known fact that,  viewed as a minimal hypersurface, $N_{r_-}$ is strictly stable. Thus, if we define  $\eta=\eta_{r_+}$ on $N_{r_+}$ (as in Proposition \ref{funcchoice}) and $\eta$ to be any constant, positive function on $N_{r_-}$, the result follows.
\end{proof}

To proceed we set
\[
\kappa_r=\frac{(n-2)m}{r^{n-1}}-r. 
\] 
It is easy to check that $|\nabla_{g_m}v|=|\kappa_r|$ on $N_{r}$ and  $|k_{r_\pm}|>0$. Hence,
\begin{equation}\label{corrfunct}
\nabla_{g_m} v=-|\kappa_{r_\pm}|\nu_\pm\quad \,\,\,{\rm on}\,\,\, N_{r_\pm}.
\end{equation}
We next consider the functional that to each metric $g$ on $M$ associates
\[
\mathcal F(g)=\int_{M}R_{g}v{\rm dvol}_{g_m}+2|\kappa_{r_-}|{\rm Area}(N_{r_-},g) + 2|\kappa_{r_+}|{\rm Area}(N_{r_+}, g). 
\]
A computation as in the proof of \cite[Proposition 9]{BMN}, which in our setting uses (\ref{statcond}) and (\ref{corrfunct}), shows that 
\[
\frac{d}{dt}\mathcal F(g(t))|_{t=0}=0,
\]
for any variation $t\in (-\epsilon,\epsilon)\mapsto  g(t)$ of $ g_m=g(0)$. 
In words, $g_m$ is a critical point for $\mathcal F$.
With this and (\ref{posit11})-(\ref{posit22}) 
at hand, 
a rather straightforward adaptation of the methods in \cite{BMN} 
provides the following important result. Here and in the following, $H_g$ denotes  the mean curvature of $\partial M$ with respect to a given ambient metric $g$.

\begin{theorem}
\label{prelprop}
For each  $n\geq 3$ and for each value of the mass parameter as in (\ref{masspar}) the gKdSS space $M=I\times N$, where $(N,h)$ is a CROSS,
carries a metric $g$ on $M$ with the following properties:
\begin{itemize}
\item $R_g>n(n-1)$ everywhere;
\item $g=g_m$ on $\partial M$;
\item $H_g>0$ on $\partial M$.
\end{itemize}
\end{theorem}

\begin{proof}
Since the detailed computation can be found in \cite{BMN}, here we merely sketch the proof. 

Surprisingly enough, the metric $g$ is constructed as a deformation of the static metric $g_m$. We first use  the test function $\eta$ in Proposition \ref{newfunc} to construct a vector field $X$ on $M$ so that 
\begin{equation}
\label{etaboun}
X=\eta\nu_{\pm},\quad ({\nabla}_{g_m})_{\nu_\pm}X=-\nabla_{h_{r_\pm}}\eta\quad {\rm on}\quad N_{r_\pm}. 
\end{equation} 
Notice that $\Lc_Xg_m=0$ along $\partial M$. We then consider the deformations of $g_m$ given by
\begin{equation}\label{twodefor}
g_0(t)=g_m+t\Lc_Xg_m, \quad g_1(t)=\varphi_t^*(g_m),
\end{equation}
where $\varphi_t$ is the flow generated by $X$. 
A computation using (\ref{twodefor}) and the fact that $g_m$ is critical for $\mathcal F$ shows that the function
\[
Q=\frac{\partial^2}{\partial t^2}R_{g_0(t)}|_{t=0}
\]
satisfies 
\begin{equation}\label{secder}
\int_M Qv{\rm dvol}_{g_m}=2\left(\int_{N_{r_-}}\eta L_{r_-}\eta {\rm dvol}_{h_{r_-}}+\int_{N_{r_+}}\eta L_{r_+}\eta{\rm dvol}_{h_{r_+}}\right), 
\end{equation}
so that 
\begin{equation}\label{mupos}
\mu=\frac{\int_M Qv{\rm dvol}_{g_m}}{\int_M v{\rm dvol}_{g_m}}>0.
\end{equation}
By (\ref{tracestat}) and Fredholm alternative we find a function $u$ on $M$ such that $\Delta_{g_m}u+nu=Q-\mu$ and $u|_{\partial M}=0$. The sought-after deformation, which is given by  
\[
g(t)=g_0(t)+\frac{1}{2(n-1)}t^2ug_m, 
\]
clearly preserves $g_m$ along $\partial M$. Moreover, 
it satisfies
\begin{equation}
\label{easycheck1}
\frac{\partial}{\partial t}R_{g(t)}|_{t=0}=0,\quad
\frac{\partial^2}{\partial t^2}R_{g(t)}|_{t=0}=Q-(\Delta_{g_m}u+nu)=\mu>0\quad {\rm on}\quad M,
\end{equation}
and 
\begin{equation}
\label{easycheck2}
\frac{\partial}{\partial t}H_{g(t)}|_{t=0}=L_{r_\pm}\eta>0 \quad {\rm on}\quad \partial M.
\end{equation}
This finishes the proof.
\end{proof}

\section{The proofs of Theorems \ref{main0} and \ref{main}}
\label{proofmain}

Theorem  \ref{prelprop} already goes a long way toward proving Theorems \ref{main0} and \ref{main}. The next  ingredient 
is an useful gluing result proved in \cite{BMN}.

\begin{theorem}
\label{gluing}\cite{BMN}
Let $M$ be a compact manifold of dimension n with boundary $\partial M$,
and let $g$ and $\tilde g$ be two smooth Riemannian metrics on M such that  $g=\tilde g$ along
$\partial M$. Moreover, assume that $H_g>H_{\tilde g}$ at each point
on $\partial M$. Given any real number $\varepsilon>0$ and any neighborhood $U$ of $\partial M$, there
exists a smooth metric $\hat g$ on M with the following properties:
\begin{itemize}
 \item $R_{\hat g}(x)\geq \min\{R_g(x),R_{\tilde g}(x)\}-\varepsilon$, for any $x\in M$;
 \item $\hat g$ agrees with $g$ outside $U$.
 \item $\hat g$ agrees with $\tilde g$ in a neighborhood of $\partial M$.
\end{itemize}
\end{theorem}

\begin{remark}
\label{regul}
{\rm Theorem \ref{gluing} can be recast as a regularity result for metrics admitting a corner along a hypersurface which preserves a given lower bound for the scalar curvature. More precisely, let $M$ be a smooth manifold and $\Sigma\subset M$ a two-sided compact hypersurface, so that $\Sigma$ divides $M$ into two domains, say $\Omega_i$, $i=1,2$. Let $g_i$ be a metric on $\Omega_i$ which extends smoothly to $\overline \Omega_i=\Omega_i\cup \Sigma$ and satisfies $R_{g_i}>c$ for some $c\in\mathbb R$. Assume further that $g_1=g_2$ on $\Sigma$ and that the mean curvature undergoes a jump as one crosses $\Sigma$, which means that $|H_{g_1}-H_ {g_2}|>0$, where both mean curvatures are computed with respect to unit normals along $\Sigma$ pointing toward the {\em same} direction. Under these conditions, Theorem \ref{gluing} implies the existence of a smooth metric $\hat g$ on $M$ satisfying $R_{\hat g}>c$ and agreeing with $g_i$ outside an arbitrarily small neighborhood of $\Sigma$.}
\end{remark}

We also need a suitable metric on $M$ to compare with the metric $g$ in Theorem \ref{prelprop}.

\begin{proposition}
\label{rotsimmet}
Any gKdSS space $M$ carries a metric $\tilde g$ such that 
$R_{\tilde g}>n(n-1)$
in a neighborhood
of $\partial M$, $\tilde g=g_m$ along $\partial M$ and $\partial M$ is totally geodesic with respect to $\tilde g$.
\end{proposition} 

\begin{proof}
We try to find $\tilde g$  along a radially symmetric conformal deformation of the type $\tilde g(t)=(1-tv^{3/2})g_m$. In a neighborhood $U$ of $\partial M$ we have 
\begin{eqnarray*}
\frac{\partial}{\partial t}R_{\tilde g(t)}|_{t=0} & = & (n-1)\left(\Delta_{g_m}v^{3/2}+nv^{3/2}\right)\\
& =  & (n-1)\left(\frac{3}{4}v^{-1/2}\kappa_r^2-\frac{n}{2}v^{3/2}\right).
\end{eqnarray*} 
Thus, $R_{\tilde g(t)}>n(n-1)$ if both $U$ and $t>0$ are chosen sufficiently small.

It is obvious that $\tilde g=g_m$ on $\partial M$. Also, if  $b_{\tilde g}$ (respectively, $b_{g_m}$) denotes the second fundamental form of the slice $N_r\subset U$ with respect to $\tilde g$ (respectively, $g_m$) we compute that 
\[
b_{\tilde g} = b_{g_m}+\frac{3}{4}\frac{tv^{1/2}\kappa_r}{\sqrt{1-tv^{3/2}}}
g_m.
\] 
By sending $r\to r_\pm$ we find that $b_{\tilde g}=0$ on $\partial M$.
\end{proof}

The proof of Theorem \ref{main0} is now completed by  
applying Theorem \ref{gluing} to the metric $\tilde g$ of Proposition \ref{rotsimmet} and the metric $g$ constructed in Theorem \ref{prelprop}. 
 
The proof of Theorem \ref{main} relies on Proposition \ref{prelprop}, Theorem \ref{gluing} and the existence of a vector field $W$ on $M$ which always points outward along $\partial M$ and has the additional properties of being  conformal in a neighborhood of $\partial M$ and Killing along $\partial M$. 
In order to exhibit such a vector field, we start by considering
\[
Y=rv(r)\frac{\partial}{\partial r}.
\] 
Notice that $Y$ never vanishes because $|Y|_{g_m}=r$. 

\begin{proposition}
\label{conffield}
The vector field $Y$ is conformal, that is, 
\begin{equation}
\label{conffieldeq}
\mathcal L_Yg_m=2vg_m.
\end{equation}
\end{proposition}

\begin{proof}
By Cartan's formula, 
\[
\mathcal L_Y\left(\frac{dr}{v}\right)=dr,
\]
and hence
\[
\mathcal L_Y\left(\frac{dr\otimes dr}{v^2}\right)=2\frac{dr\otimes dr}{v}.
\]
Thus, 
\begin{eqnarray*}
\mathcal L_Y g_m & = & \mathcal L_Y\left(\frac{dr\otimes dr}{v^2}\right)+ Y(r^2)h\\
& = & 2 \frac{dr\otimes dr}{v} +2r^2vh\\
& =  & 2vg_m,
\end{eqnarray*}
as desired. 
\end{proof}

Notice that, as the proof above makes it clear, the conformal character of $Y$ holds true irrespective of the nature of $v=v(r)$.

We now fix a radially symmetric neighborhood $U_\pm$ of $N_{r_\pm}$ and take a cut-off function $\psi_\pm$ which holds 1  on $U_\pm$ and $0$ on $U_\mp$. We then consider the vector field $W=(\psi_+-\psi_-)Y$. Thus, $W$ is conformal on $U=U_+\cup U_-$ and  points outward along $\partial M$.
Moreover, $\Lc_Wg_m=0$ on $\partial M$. Thus, $W$ has all the properties listed above.

In the following we denote by $\Psi_t$ the flow of $W$ at time $t$. Notice that for each $t>0$ small enough, $\Psi_t$ carries a radially symmetric region of the type $\{r_-^t\leq r\leq r_+^t\}$ onto $M$. Here, $r_-<r_-^t<r_+^t<r_+$. 

Given $\delta_+,\delta_->0$ small enough, we define $r_{i\delta_\pm}$, $i=1,2,3$, by the conditions:
\begin{itemize}
\item  $r_-< r_{\delta_-}<r_{2\delta_-}<r_{3\delta_-}<r_*<r_{3\delta_+}< r_{2\delta_+}<r_{\delta_+}<r_+$; 
\item  $v(r_{i\delta_\pm})=i\delta_\pm$. 
\end{itemize}
We also set 
$M_{i\delta_\pm}=\{r_{i\delta_-}\leq r\leq r_{i\delta_+}\}$. 
If $\delta_+>0$ is small enough there exists $\delta_->0$ and
$t_{\delta_+}>0$ so that 
$\Psi_{t_{\delta_+}}$ carries $M_{2\delta_\pm}$ onto $M$. It is clear that $\delta_-\to 0$ and $t_{\delta_+}\to 0$ as $\delta_+\to 0$. 
For each such $\delta_+>0$ we define a metric $\tilde g_{\delta_+}$ on $M$ satisfying:
\[
\tilde g_{\delta_+}=\left\{\begin{array}{lcr}
g_m & {\rm on} & M\setminus \stackrel{\circ}{M}_{\delta_\pm}\\
(1-e^{-\frac{1}{v-\delta_+}})^{\frac{4}{n-2}}g_m &{\rm on} & 
\{r_{3\delta_+}<r< r_{\delta_+}\}\\
(1-e^{-\frac{1}{v-\delta_-}})^{\frac{4}{n-2}}g_m &{\rm on} & 
\{r_{\delta_-}<r< r_{3\delta_-}\}
\end{array}
\right. 
\]  
Then, as in \cite[Proposition 21]{BMN}, we have

\begin{proposition}
\label{prop21}
If $\delta_+>0$ is small enough then $R_{\tilde g_{\delta_+}}>n(n-1)$ in $\stackrel{\circ}{M}_{\delta_\pm} \setminus  \stackrel{\circ}{M}_{3\delta_\pm}$.
\end{proposition}

Clearly, we may choose $\delta_+$ small enough so that $M\setminus {M}_{3\delta_\pm}\subset U$. Therefore, on $M_{2\delta_\pm}\cap U$ we have $\Psi_{t_{\delta_+}}^*g_m=\theta g_m$, where $\theta=\theta^{(t_{\delta_+})}$ is a positive, radially symmetric function.  We now define a metric $g_{\delta_+}$ on ${M}_{2\delta_\pm}$ by $g_{\delta_+}=\Theta\Psi^*_{t_{\delta_+}}g$, 
where $g$ is the metric in Theorem \ref{prelprop} and 
$\Theta$ is a radially symmetric function on $M$ satisfying:
\begin{itemize}
\item $\Theta=(1-e^{-\frac{1}{\delta_+}})^{\frac{4}{n-2}}\theta(r_{2\delta_+})^{-1}$ on $\{r_{2\delta_+}\leq r\leq r_+\}$;
\item $\Theta=(1-e^{-\frac{1}{\delta_-}})^{\frac{4}{n-2}}\theta(r_{2\delta_-})^{-1}$ on $\{r_-\leq r\leq r_{2\delta_-}\}$;
\item $\Theta$ linearly interpolates between these regions. 
\end{itemize}

\begin{proposition}
\label{propcoinc}
We have $g_{\delta_+}=\tilde g_{\delta_+}$ on $\partial M_{2\delta_\pm}$.
\end{proposition}

\begin{proof}
We use that $g=g_m$ on $\partial M$ to check that, on $\partial M_{2\delta_\pm}$,
\begin{eqnarray*}
g_{\delta_+} & = & (1-e^{-\frac{1}{2\delta_\pm-\delta_\pm}})^{\frac{4}{n-2}}\theta(r_{2\delta_\pm})^{-1}\Psi^*_{t_{\delta_+}}g_m\\
& = & (1-e^{-\frac{1}{\delta_\pm}})^{\frac{4}{n-2}}\theta(r_{2\delta_\pm})^{-1}\theta(r_{2\delta_\pm})g_m\\
 &=& \tilde g_{\delta_+},
\end{eqnarray*}
as desired. 
\end{proof}

We now have all the ingredients needed to complete the proof of Theorem \ref{main}. Since $\tilde g_{\delta_+}\to g_m$ as $\delta_+\to 0$ on $\partial M_{\delta_\pm}$, we may choose $\delta_+>0$ small enough so that 
\[
\sup_{\partial M_{2\delta_\pm}} H_{\tilde g_{\delta_+}} < \inf_{\partial M} H_{g}, 
\]
Since $\mathcal L_Wg_m=0$ on $\partial M$, we have  that $\theta=\theta^{(t_{\delta_+})}\to 1$ on $\partial M_{2\delta_\pm}$ as $\delta_+\to 0$, which implies that $g_{\delta_+}\to g$ as $\delta_+\to 0$.
Thus, we have 
\[
\sup_{\partial M_{2\delta_\pm}} H_{\tilde g_{\delta_+}} < \inf_{\partial M_{2\delta_\pm}} H_{g_{\delta_+}}, 
\]
and $R_{g_{\delta_+}}>n(n-1)$ for any $\delta_+$ small enough. 
By Theorem \ref{gluing} and Proposition \ref{prop21}, there exists a metric $\hat g$ on $M_{2\delta_\pm}$ with $R_{\hat g}>n(n-1)$ and which agrees with $\tilde g_{\delta_+}$ in a neighborhood of $\partial M_{2\delta_\pm}$. Thus, $\hat g$ extends to a metric on $M$, still denoted $\hat g$, which satisfies  $R_{\hat g}\geq n(n-1)$ and 
agrees with $\tilde g_{\delta_+}$ in the region $M\setminus M_{2\delta_+}$. In particular, $\hat g=g_m$ on $M\setminus M_{\delta_\pm}$. This concludes the proof of Theorem \ref{main}.

\section{Some examples}\label{equalinf}

In this section we briefly indicate  how the metrics above can be used as building blocks in the construction of solutions of Einstein field equations with $\Lambda>0$ and whose geometry is  precisely controlled along the event horizon and in a neighborhood of spatial infinity.
The examples presented below should be compared to the vacuum solutions obtained by sophisticated gluing techniques in \cite{CP} and \cite{CPP}; see also \cite{D}.

We first describe the model solution. 
We start by gluing two copies of the gKdSS space $(M,g_m)$ along its common boundary $N_{r_+}$, thus obtaining a double gKdSS space with boundary two copies of $N_{r_-}$. We now 
glue a countable collection of such manifolds along their common boundary $N_{r_-}$, thus producing  
a one-parameter family of periodic metrics which   interpolate between an infinite sequence of unit spheres and an infinite cylinder\footnote{These Delaunay-type metrics  play a central role in questions  related to the singular Yamabe problem; see  \cite{CP} and the references therein.}. The maximal Cauchy development of this  IDS yields  a static vacuum solution of Einstein field equations with positive cosmological constant which displays an infinite number of alternating event and cosmological horizons all the way up to spatial infinity and whose boundary is isometric to $\mathbb R\times N_{r_-}$. We shall refer to such a solution as a {\em gKdSS space-time}. 
Now, if $N$ is a CROSS we can perform this same construction starting off with the manifold $(M,g)$  in Theorem \ref{main}. This provides a non-vacuum solution which satisfies the dominant energy condition and displays the same horizon behavior  as the gKdSS space-time, with alternating event and cosmological horizons. Moreover, since in the non-vacuum region the metric in Theorem \ref{main} can be taken an arbitrarily small perturbation of $g_m$, we can arrange so that the corresponding Cauchy development asymptotes the gKdSS space at spatial infinity with any prescribed rate. More drastically, we can stabilize the gluing procedure so as to obtain a non-vacuum solution which agrees with the gKdSS space-time along the event horizon and in a whole neighborhood of spatial infinity. This yields physically interesting solutions which are finite scale perturbations of a gKdSS.          

We can also provide similar examples by starting with the manifold  produced in Theorem \ref{prelprop}. This time the resulting metric admits  corners along the glued boundaries. However, since the scalar curvature of the building blocks is strictly larger than $n(n-1)$ and a jump for the mean curvature certainly occurs along the corners, the appropriate variant of Theorem \ref{gluing} as in Remark \ref{regul} can be used to obtain a smooth metric $g$ on the infinite cylinder with $R_g>n(n-1)$. This yields a nowhere vacuum IDS satisfying the dominant energy condition. Notice that we can always arrange for the boundary to be totally geodesic by attaching to it the appropriate  double of the manifold appearing in Theorem \ref{main0} and regularizing the resulting corner.   
These examples are of interest in connection with a rigidity result established by M\'aximo-Nunes \cite{MN}. More precisely, these authors show that if an embedded minimal sphere $\Sigma$ in a $3$-manifold satisfying $R\geq 6$ is strictly stable and locally maximizes the Hawking mass then a neighborhood of $\Sigma$ is isometric to a neighborhood of $\mathbb S^2_{r_-}$ in a dSS space. Since, as before, we may assume that the solution above asymptotes a dSS space at spatial infinity, it clearly carries a stable minimal sphere in the asymptotic region. However, since the underlying metric in each building block is not even locally isometric to the dSS space, this sphere cannot satisfy the above mentioned assumptions. In particular, if it is strictly stable then it does not locally maximize the Hawking mass.

Starting from the examples in the previous paragraph we may appeal to  a gluing result due to Delay \cite[Theorem 4.1]{D} to obtain  solutions which agree with the gKdSS space in a neighborhood of infinity and are non-vacuum in the complement of this neighborhood, while keeping the given bounds on the scalar curvature. Again, this yields physically relevant solutions which only differ from a gKdSS space at  finite scales. Finally, we remark that the class of examples above can be substantially enlarged  by performing connected sums around non-vacuum points, as explained in \cite{GL}.  This yields examples with  disconnected event horizons and fairly complicated topology.

\end{document}